\documentclass{kms-b}

\newcommand{\publname}{Bull.~Korean Math.~Soc.}
\issueinfo{}% volume number
  {}%        % issue number
  {}%        % month
  {}%     % year
\pagespan{1}{}
\copyrightinfo{}%              % copyright year
  {Korean Mathematical Society}% copyright holder
\def\r{\mathbb{R}}
 \def\s{\mathbb{S}} 
 
\def\d{\mathsf{D}}
\def\a{\mathsf{a}}

\usepackage{graphicx}
\allowdisplaybreaks

\theoremstyle{plain}
\newtheorem{theorem}{Theorem}[section]

\newtheorem{lemma}[theorem]{Lemma}
\newtheorem{corollary}[theorem]{Corollary}

\theoremstyle{definition}

\theoremstyle{remark}
\newtheorem{remark}[theorem]{Remark}

\begin{document}

\title[Stability of graphs in the Euclidean space with density]
{Stability of vertical and radial graphs in the Euclidean space with density}

\author[R. L\'opez]{Rafael L\'opez}
\address{Rafael L\'opez \\ Department of Geometry and Topology \\ University of Granada \\ 18071 Granada, Spain}
\email{rcamino@ugr.es}
%\thanks{This work was financially supported by KRF 2003-041-C20009}


\begin{abstract}
   It is proved that vertical graphs   and radial graphs are strongly stable for a certain type of densities in Euclidean space $\r^{n+1}$. Particular cases of these densities include translators,   expanders and singular minimal hypersurfaces. Using techniques of calibrations, it is also proved that for densities depending on a spatial coordinate, stationary vertical graphs are weighted minimizers   in a certain class of hypersurfaces.
\end{abstract}

\maketitle

%%%%%%%%%%%%%%%%%%%%
\section{Introduction  and preliminaries}
%%%%%%%%%%%%%%%%%%
 
 In this paper we will investigate  under what conditions vertical graphs (graphs on hyperplanes) and radial graphs (graphs on spheres) are stable and minimizers  in   the Euclidean space $\r^{n+1}$ endowed with a density $e^\phi$, where $\phi\in C^\infty(\r^{n+1})$. Consider in $ \r^{n+1}$ the Euclidean metric $\langle,\rangle$, and let  $dA$ and $dV$  denote the  Euclidean area and volume forms, respectively. The weighted area (perimeter) and  volume for the density $e^\phi$  are defined by 
$$ d A_\phi= e^{\phi} dA,\quad \, d V_\phi = e^{\phi}  dV,$$
respectively   \cite{gr,mo}.    Stable hypersurfaces are minimizers of the weighted area functional up to second order and thus, they are candidates to be global minimizers of the weighted area.   

Let $\Sigma$ be a smooth, oriented hypersurface immersed in $\r^{n+1}$, and let $N$ and $H$ denote the Gauss map and the mean curvature of $\Sigma$, respectively. Let $\{\Sigma_t:|t|<\epsilon\}$ be a variation of $\Sigma$, $\Sigma_0=\Sigma$,  with compactly supported variational vector field $\xi$. Let $A_\phi(t)$ and $V_\phi(t)$ be the weighted area and the weighted volume of $\Sigma_t$. The   formulas of the first variation of  $A_\phi$ and   $V_\phi$ are
\begin{equation}\label{aa}
A'_\phi(0)=-\int_\Sigma u\left(nH-\langle\overline{\nabla}\phi,N\rangle\right)  \, d A_\phi,\quad \,  V'_\phi(0) = \int_\Sigma u\,  dA_\phi,
 \end{equation}
where  $\overline{\nabla}$ is the Euclidean gradient in $\r^{n+1}$ and $u=\langle\xi,N\rangle\in C_0^\infty(\Sigma)$ is the normal component of $\xi$. The function $H_\phi$  defined by
\begin{equation}\label{h}
H_\phi=nH-\langle  \overline{\nabla}\phi,N \rangle
\end{equation}
is called the {\it weighted mean curvature} of $\Sigma$. As a consequence,  $\Sigma$ is a critical point of $A_\phi$ for all weighted volume preserving variations if and only if $H_\phi=\lambda$ is constant, $\lambda\in\r$. We  then say that $\Sigma$ is a {\it $\phi$-stationary hypersurface}.   A special case is $\lambda=0$. Then     $H_\phi=0$ if and only if $\Sigma$ is a critical point of the $A_\phi$ for any variation of $\Sigma$ and we say that $\Sigma$ is a {\it $\phi$-minimal hypersurface}.   

In this paper, we investigate  two types of densities. 
  
\begin{enumerate}
\item Radial densities. The function $\phi$ only depends on the distance to the origin $O$ of $\r^{n+1}$. If $r=|x|^2$, $x=(x_1,\ldots,x_{n+1})\in\r^{n+1}$, then $\phi=\phi(r)$. Since $\overline{\nabla}\phi=2\phi'(r)x$,   the weighted mean curvature is given by 
$$H_\phi=nH-2\phi'(r)\langle N,x\rangle.$$
 Interesting cases are self-shrinkers and expanders, where   $\phi(x)=\epsilon |x|^2/4$, with $\epsilon=1$ (expander) and $\epsilon=-1$ (self-shrinker). Then
$$H_\phi=nH-\frac{\epsilon}{2}\langle N,x\rangle.$$
Both  hypersurfaces are particular solutions of the mean curvature flow in $\r^{n+1}$ which   evolve by  expanding or shrinking homothetically \cite{hui1,hui2}. Also constant mean curvature  hypersurfaces (cmc to abbreviate) appear when $\phi=1$. 

Very recently, Dierkes and Huisken have studied stability for a large family of radial densities $\phi=\phi(r)$, including of type $\phi(x)=|x|^p$, with $p\in\r$ \cite{dh2}.   
\item Densities depending on a spatial coordinate.  The function $\phi$ is $\phi(x)=\phi(x_{n+1})$. If  $\{\a_i:1\leq i\leq n+1\}$ is the canonical basis of $\r^{n+1}$, then $\overline{\nabla}\phi=\phi'(t)\a_{n+1}$. Thus
$$H_\phi=nH-\phi'(t)\langle N,\a_{n+1}\rangle=nH-\phi'(t)N_{n+1},$$
 where    $N=(N_1,\ldots,N_{n+1})$. If $\phi=1$, then $\phi$-stationary hypersurfaces are cmc hypersurfaces. Another two examples of this type of densities are the following:
\begin{enumerate}
\item 
$\lambda$-translators.  Here $\phi(t)=t$.  When $\lambda=0$, $0$-translators are the translating solitons   of the mean curvature flow in $\r^{n+1}$ which evolve by translations \cite{il}. The weighted mean curvature is   
$$H_\phi=nH-N_{n+1}.$$
\item Singular minimal hypersurfaces. Here $\phi(t)=\alpha\log(t)$,  where $\alpha\in\r$ and it is assumed  that $\Sigma$ is included in the halfspace $\r^{n+1}_{+}=\{x\in\r^{n+1}:x_{n+1}>0\}$. The weighted mean curvature is  
$$H_\phi=nH-  \alpha \frac{N_{n+1}}{x_{n+1}}.$$
A $\phi$-stationary hypersurface is called $\alpha$-singular minimal hypersurface. If  $\alpha=1$, then $\phi$-minimal hypersurfaces   correspond  with the $n$-dimensional analogue of the catenary  \cite{bht,di1,di2,dh}. If $\alpha=-n$, then $\phi$-minimal hypersurfaces  are minimal hypersurfaces of the hyperbolic space when $\r^{n+1}_{+}$ is endowed with the hyperbolic metric.  
 \end{enumerate}
\end{enumerate}
 
A $\phi$-stationary hypersurface $\Sigma$ is said to be  {\it strongly $\phi$-stable} (resp.  {\it $\phi$-stable}) if  $A''_\phi(0)\geq 0$ for all variations of $\Sigma$ (resp. all weighted volume preserving variations).   
    The expression of the second derivative of $A_\phi(t)$   for compactly supported normal variations of $\Sigma$ is 
\begin{equation}\label{ss}
A_\phi ''(0)=-\int_\Sigma u \left(\Delta u+\langle\overline{\nabla}\phi,\nabla u\rangle+(|A|^{2}-\overline{\nabla}^2\phi(N,N)\big)u\right)\, d A_\phi,
\end{equation}
where $u=\langle\xi,N\rangle\in C_0^\infty(\Sigma)$. See   \cite{ba,cr2}. Here $|A|^2$ denotes the square of the norm of the second fundamental form $A$ of   $\Sigma$,   $\nabla$ and $\Delta$ are the gradient and the Laplacian operators computed on $\Sigma$ with the induced Euclidean metric  and $\overline{\nabla}^2$ is the Hessian operator in $ \r^{n+1}$.  In the case that the admissible variations preserve the weighted volume of $\Sigma$,  the function $u$ satisfies the extra condition  $\int_\Sigma u\, d A_\phi=0$.   The term $-\overline{\nabla}^2\phi$ in \eqref{ss} coincides with the Bakry-\'Emery Ricci tensor.  Results of $\phi$-stability for different conditions on the Bakry-\'Emery Ricci tensor  appear in  \cite{al,ca2,cr,ch,ir,li}. In the particular case of translating solitons and self-shrinkers, see \cite{cl,im,irs,sh}.

In Section \ref{sec2} we will derive $\phi$-stability of graphs for radial densities and densities depending on a spatial coordinate.   Here we distinguish two types of graphs: vertical graphs, which are graphs on domains of a hyperplane $\Pi$ of $\r^{n+1}$, and  radial graphs on domains of  the unit sphere $\s^n$ centered  at $O\in\r^{n+1}$. 

For radial densities $e^\phi$, $\phi=\phi(r)$, we prove in Theorem \ref{t1} that $\phi$-stationary vertical graphs are strongly $\phi$-stable if the Bakry-\'Emery Ricci tensor is nonpositive. This hypothesis is equivalent to $2\phi'(r)+\phi''(r)\langle N,x\rangle^2\geq 0$ on the graph. In Theorem \ref{t2}, it is proved  that   radial graphs are strong $\phi$-stable if $\phi'(r)+\phi''(r)\langle N,x\rangle^2\geq 0$. In particular,  vertical graph expanders and   radial graph expanders are strongly $\phi$-stable (Corollaries \ref{ce1} and \ref{ce2}, respectively). 

For densities $e^\phi$ depending on a spatial coordinate, $\phi=\phi(x_{n+1})$, strong $\phi$-stability for vertical graphs is assured if $\phi$ is convex (Theorem \ref{t4}). In contrast, strong $\phi$-stability for radial graphs is difficult to study in general. In the particular case of $\alpha$-singular minimal surfaces, we are able to prove strong  $\phi$-stability of radial graphs if $\alpha\leq 0$ (Theorem \ref{t5}). 

In Section \ref{sec3}, we consider the problem of global minimization of the weighted area $A_\phi$ when the density  depends only on a spatial coordinate. If $\phi=\phi(x_{n+1})$ is convex, then it is proved that vertical graphs on a domain $U$ of the hyperplane $x_{n+1}=0$ are minimizers in the class of hypersurfaces in $U\times\r$ with the same boundary and the same homology class (Theorem \ref{estable1}).

%%%%%%%%%%%%%%%%%%%
\section{Stability of vertical graphs and radial graphs}\label{sec2}
%%%%%%%%%%%%%%

 We express the notion of $\phi$-stability in terms of a quadratic form defined in the space of compactly supported functions $C_0^\infty(\Sigma)$. Consider the expression of $A_\phi''(0)$ given in \eqref{ss}. The $\phi$-divergence of a vector field $X\in\mathfrak{X}(\Sigma)$ is  defined by $\mathrm{div}_\phi X=\mathrm{div}X+\langle\overline{\nabla}\phi,X\rangle$. Thus the $\phi$-Laplacian of $u  \in C^\infty(\Sigma)$   is defined by 
\begin{equation}\label{div}
\Delta_\phi u=\mathrm{div}_\phi(\nabla u)=\Delta u+\langle \overline{\nabla}\phi,\nabla u\rangle.
\end{equation}
The parenthesis in \eqref{ss} defines an elliptic operator called the {\it $\phi$-Jacobi operator} $L_\phi$, $L_\phi\colon C_0^\infty(\Sigma)\to C_0^\infty(\Sigma)$, which is given by 
$$L_\phi [u] = \Delta_\phi u+(|A|^{2}-\overline{\nabla}^2\phi(N,N)\big)u.$$ 
 Both operators $\Delta_\phi$ and  $L_\phi$ are not self-adjoint with respect to the $L^2$-inner product, but they are with respect to the weighted inner product $\int_\Sigma uv\, dA_\phi$ for $u,v\in C_0^\infty(\Sigma)$. Therefore, the expression of $A_\phi''(0)$ in \eqref{ss} defines  the quadratic form  
\begin{equation}\label{qq}
Q_\phi[u]=-\int_\Sigma u\cdot L_\phi[u]\, d A_\phi,\quad u\in C^\infty_0(\Sigma).
\end{equation}
This implies that  $\Sigma$ is strongly $\phi$-stable if and only if $Q_\phi[u]\geq 0$ for all $u\in C_0^\infty(\Sigma)$. 

In this section, we will  establish conditions that ensure that  vertical graphs  and radial graphs  of $\r^{n+1}$ are strongly $\phi$-stable for radial  densities and densities depending on a spatial coordinate. If $\Pi$ is  a hyperplane with $\Pi=\a^\top$, $\lvert \a\rvert=1$, a  vertical graph $\Sigma\subset\r^{n+1}$ on a domain of $\Pi$ satisfies that the function $\langle N,\a\rangle$ does not vanish on $\Sigma$. Analogously, on a radial graph $\Sigma$ on a domain of $\s^n$,  the support function $ \langle N,x\rangle$ does not vanish. As a  consequence, for vertical graphs (resp. radial graphs) any straight-line orthogonal to $\Pi$ (resp. straight-line  starting from $O$) meets the hypersurface once at most.

We begin with the study of vertical graphs. First, we compute the $\phi$-Laplacian of the function $\langle N,\a\rangle$. 

\begin{lemma} Let $\a\in\r^{n+1}$, $\lvert \a\rvert=1$. If $\Sigma$ is a $\phi$-stationary surface, then the function $g=\langle N,\a\rangle$ satisfies
\begin{equation}\label{n1}
\Delta_\phi g+|A|^2g=-\sum_{i=1}^n\langle\d_{e_i} \overline{\nabla}\phi,N\rangle \langle e_i,\a\rangle,
\end{equation}
where $\d$ is the Levi-Civita connection in $\r^{n+1}$ and $\{e_i:1\leq i\leq n\}$ is a local  orthonormal frame on  $\Sigma$.
\end{lemma}

\begin{proof}
It is known that in any hypersurface, the Gauss map $N$ satisfies 
\begin{equation}\label{no}
\Delta N+|A|^2N+\nabla(nH)=0.
\end{equation}
 Since $H_\phi$ is constant, $H_\phi=\lambda$, from \eqref{h} we have $nH=\langle\overline{\nabla}\phi,N\rangle+\lambda$. Thus
\begin{equation*}
\nabla (nH) =\nabla\langle \overline{\nabla}\phi,N\rangle=\sum_{i} \langle\nabla \langle\overline{\nabla}\phi,N\rangle,e_i\rangle e_i=\sum_i\left(\langle\d_{e_i} \overline{\nabla}\phi,N\rangle+\langle\overline{\nabla}\phi,\d_{e_i}N\rangle\right)  e_i . 
\end{equation*}
With this identity, equation \eqref{no} is written as 
\begin{equation*} 
\Delta N+|A|^2N+ \sum_i \langle\overline{\nabla}\phi,\d_{e_i}N\rangle     e_i  =-\sum_{i} \langle\d_{e_i} \overline{\nabla}\phi,N\rangle e_i.
\end{equation*}
Multiplying by $\a$, we obtain 
\begin{equation}\label{c11}
\Delta g+|A|^2g +\sum_i \langle\overline{\nabla}\phi,\d_{e_i}N\rangle   \langle e_i,\a\rangle=-\sum_{i} \langle\d_{e_i} \overline{\nabla}\phi,N\rangle \langle e_i,\a\rangle.
\end{equation}
On the other hand, we compute the term $\langle\overline{\nabla}\phi,\nabla g\rangle$ in \eqref{div}.  Since $\d_{e_i}N$ is tangent to $\Sigma$,   we have
$$\langle \nabla g,e_i\rangle=\langle\d_{e_i}N,\a\rangle=\sum_j\langle \d_{e_i}N,\a_j\rangle\langle e_j,\a\rangle.$$
Using that the Weingarten endomorphism is self-adjoint, then
\begin{equation}\label{c12}
\begin{split}
\langle\overline{\nabla}\phi,\nabla g\rangle&=\sum_{i} \langle \overline{\nabla}\phi,e_i\rangle\langle \nabla g ,e_i\rangle=\sum_{i,j}   \langle \overline{\nabla}\phi,e_i\rangle\langle\d_{e_i} N,e_j\rangle\langle e_j ,\a\rangle\\
&=\sum_{i,j}  \langle \overline{\nabla}\phi,e_i\rangle\langle e_i,\d_{e_j} N\rangle\langle e_j ,\a\rangle=
\sum_{i}   \langle \overline{\nabla}\phi, \d_{e_i} N\rangle\langle e_i ,\a\rangle.
\end{split}
\end{equation}
Combining   identities \eqref{c11} and \eqref{c12}, we obtain \eqref{n1}.
\end{proof}
Let's introduce the notation
$$h=\langle N,x\rangle$$
for  the support function.  The first result  establishes conditions for a vertical graph to be strongly $\phi$-stable. As it was announced, we will investigate two types of densities. We begin with the study of radial densities.

\begin{theorem} \label{t1}
Let $e^\phi$ be a radial density, with $\phi=\phi(r)$,    $r=|x|^2$. Let $\Sigma$ a  $\phi$-stationary vertical graph. If $\phi'(r)+2 \phi''(r)h^2\geq 0$ on $\Sigma$, then   $\Sigma$ is  strongly $\phi$-stable.  
\end{theorem}

\begin{proof}
Let $\Sigma$ be a $\phi$-stationary graph on a domain of a hyperplane $\Pi$.  After a change of coordinates, we can assume that $\Pi$ is the plane $x_{n+1}=0$. Since $\phi=\phi(r)$, we have $\overline{\nabla}\phi=2\phi'(r) x$   and 
\begin{equation}\label{he}
\overline{\nabla}^2\phi(N,N)=2(\phi'(r)+2\phi''(r) \langle N,x\rangle^2)=2(\phi'(r)+2\phi''(r) h^2).
\end{equation}
We know that the function $ N_{n+1}=\langle N,\a_{n+1}\rangle$ has sign on $\Sigma$.    Since $\langle \d_{e_i}\overline{\nabla}\phi,N\rangle=2\phi'(r)\langle e_i,N\rangle=0$,  the right hand-side of  \eqref{n1} vanishes and thus
\begin{equation}\label{lu}
L_\phi[N_{n+1}]=-\overline{\nabla}^2\phi(N,N)N_{n+1}= -2(\phi'(r)+2\phi''(r) h^2)N_{n+1}.
\end{equation}
 Let  $u\in C_0^\infty(\Sigma)$ be arbitrary and we    prove that $Q_\phi[u]\geq 0$.   Let $v=u/N_{n+1}\in C_0^\infty(\Sigma)$. For $u=vN_{n+1}$, we have
\begin{equation*}
\begin{split}
\Delta u&=v\Delta N_{n+1}+N_{n+1}\Delta v+2\langle\nabla v,\nabla N_{n+1}\rangle,\\
\langle\overline{\nabla}\phi,\nabla N_{n+1}\rangle&=N_{n+1}\langle \overline{\nabla}\phi,\nabla v \rangle+v\langle \overline{\nabla}\phi,\nabla N_{n+1} \rangle.
\end{split}
\end{equation*}
Then
\begin{equation}\label{l4}
\begin{split}
L_\phi[u]&=v\cdot L_\phi[N_{n+1}]+N_{n+1}\left(\Delta v +\langle \overline{\nabla}\phi,\nabla v\rangle\right)+    2\langle\nabla v,\nabla N_{n+1}\rangle \\
&=v\cdot L_\phi[N_{n+1}]+N_{n+1}\Delta_\phi v+2\langle\nabla v,\nabla N_{n+1}\rangle.
\end{split}
\end{equation}
Therefore
\begin{equation}  \label{q1}
\begin{split}
Q_\phi[u]&=-\int_\Sigma u\cdot L_\phi[u]\, dA_\phi=-\int_\Sigma v N_{n+1}\cdot L_\phi[u]\, dA_\phi\\
&=-\int_\Sigma\left( v^2  N_{n+1}\cdot L_\phi[N_{n+1}]+vN_{n+1}^2\Delta_\phi v+2v N_{n+1}  \langle\nabla v,\nabla N_{n+1}\rangle\right) \, d A_\phi.
\end{split}
\end{equation}
For the vector field $v N_{n+1}^2\nabla v$, we have   
$$\mathrm{div}_\phi(v N_{n+1}^2\nabla v)=vN_{n+1}^2\Delta_\phi v+N_{n+1}^2\lvert \nabla v\rvert^2+2vN_{n+1}  \langle\nabla N_{n+1},\nabla v\rangle.$$
Using the $\phi$-divergence theorem   in weighted manifolds  and because $v$ is compactly supported in $\Sigma$, 
$$\int_\Sigma (vN_{n+1}^2\Delta_\phi v +2vN_{n+1} \langle\nabla v,\nabla N_{n+1}\rangle )\, d A_\phi=-\int_\Sigma N_{n+1}^2\lvert\nabla v\rvert^2\, d A_\phi.$$
Putting into \eqref{q1} and taking into account \eqref{lu}, we find
\begin{equation}\label{l5}
\begin{split}
Q_\phi[u]&=\int_\Sigma N_{n+1}^2\lvert\nabla v\rvert^2\, d A_\phi-\int_\Sigma v^2  N_{n+1} \cdot L_\phi[N_{n+1}] \, d A_\phi\\
&=\int_\Sigma N_{n+1}^2\lvert\nabla v\rvert^2\, d A_\phi+2\int_\Sigma v^2N_{n+1}^2(\phi'(r)+2\phi''(r) h^2) \, d A_\phi\geq 0,
\end{split}
\end{equation}
as desired.
\end{proof}

\begin{remark}\label{rem1}
Monotonically increasing convex functions   $\phi$    satisfy  $\phi'(r)+2\phi''(r)h^2\geq 0$.
\end{remark}

Let us observe that for self-shrinkers and expanders, $\phi(r)=\epsilon r/4$. In particular, $ \phi'(r)+2\phi''(r)h^2=\epsilon/2$.   Consequently, we have proved the following result. 

\begin{corollary}\label{ce1} Vertical graph expanders are strongly $\phi$-stable.  
\end{corollary}

We now study $\phi$-stability of radial graphs. We know that the support function  $h=\langle N,x\rangle$  does not vanish on a radial graph. First we  compute the $\phi$-Jacobian operator of   $h$. 

\begin{lemma} If $\Sigma$ is a $\phi$-stationary surface, then the support function $h$ satisfies
\begin{equation}\label{h3}
\Delta_\phi h+\lvert A\rvert^2h=-nH-\sum_{i=1}^n  \langle\d_{e_i} \overline{\nabla}\phi,N\rangle \langle e_i,x\rangle,
\end{equation}
where $\{e_i:1\leq i\leq n\}$ is a local orthonormal frame on $\Sigma$.
\end{lemma}

\begin{proof} 
Let $x_j=\langle x,\a_j\rangle$ and $N_j=\langle N,\a_j\rangle$. Because $h=\sum_{j=1}^{n+1}x_jN_j$,   we have 
\begin{equation}\label{hs}
 \Delta h=\sum_jN_j\Delta x_j+2\sum_j\langle \nabla N_j,\nabla x_j\rangle+\sum_j x_j\Delta N_j.
 \end{equation}
 Since $\Delta x_j=nH N_j$, then $\sum_jN_j\Delta x_j=nH$. On the other hand,  
 \begin{equation*}
 \begin{split}
 \sum_j\langle \nabla N_j,\nabla x_j\rangle&=-\sum_{i,j}  \langle \nabla N_j,e_i\rangle\langle\nabla x_j,e_i\rangle=
 \sum_{i,j}\langle\d_{e_i}N,\a_j\rangle\langle\a_j,e_i\rangle\\
 &=\sum_i\langle \d_{e_i}N,e_i\rangle=-nH.
 \end{split}
 \end{equation*}
Using this identity and \eqref{n1}, equation \eqref{hs} is
 \begin{equation*}
\begin{split}
\Delta h&=-nH+\sum_j x_j\Delta N_j\\
 &=-nH-\sum_j x_j\left(\lvert A\rvert^2 N_j+\langle \overline{\nabla}\phi,\nabla N_j\rangle+\sum_{i} \langle\d_{e_i} \overline{\nabla}\phi,N\rangle \langle e_i,\a_j\rangle\right)\\
  &=-nH- \lvert A\rvert ^2 h-\sum_j x_j\langle \overline{\nabla}\phi,\nabla N_j\rangle-\sum_{i} \langle\d_{e_i} \overline{\nabla}\phi,N\rangle \langle e_i,x\rangle.
 \end{split}
 \end{equation*}
 Then
 \begin{equation}\label{a0}
 \Delta h+\lvert A\rvert^2h +\sum_j x_j\langle \overline{\nabla}\phi,\nabla N_j\rangle=-nH-\sum_{i} \langle\d_{e_i} \overline{\nabla}\phi,N\rangle \langle e_i,x\rangle.
 \end{equation}
 For the computation of $\Delta_\phi h=\Delta h+\langle \overline{\nabla}\phi,\nabla h\rangle$, we calculate the second term. We have  
 \begin{equation}\label{a1}
\langle \overline{\nabla}\phi,\nabla h\rangle=\sum_j x_j\langle \overline{\nabla}\phi,\nabla N_j\rangle+\sum_j N_j\langle \overline{\nabla}\phi,\nabla x_j\rangle,
\end{equation}
 and
\begin{equation}\label{a2}
\begin{split}
\sum_j N_j\langle \overline{\nabla}\phi,\nabla x_j\rangle&=\sum_jN_j\langle\overline{\nabla}\phi,\a_j-N_jN \rangle\\
&=\sum_j N_j\langle \overline{\nabla}\phi,\a_j\rangle-(\sum_j N_j^2)\langle \overline{\nabla}\phi,N\rangle\\
&=0.
\end{split}
\end{equation}
 Combining \eqref{a0}, \eqref{a1} and \eqref{a2}, we obtain \eqref{h3}.
\end{proof}

The first result is obtained for $\phi$-minimal radial graphs. 
 
 \begin{theorem}\label{t2}
  Let $e^\phi$ be a radial density, with $\phi=\phi(r)$,  $r=\lvert x\rvert^2$. Let $\Sigma$ be a $\phi$-minimal  radial graph. If $\phi'(r)+ \phi''(r)h^2\geq 0$ on $\Sigma$, then   $\Sigma$ is strongly $\phi$-stable.  
\end{theorem}

\begin{proof}
First consider the general case $H_\phi=\lambda$,  where $\lambda\in\r$ is a constant. 
Since $\phi$ is radial, from  \eqref{h}, \eqref{he} and \eqref{h3}, we obtain
\begin{equation}\label{f1}
L_\phi[h]=-nH-\nabla^2\phi(N,N) h=  -4h(\phi'(r) +\phi'' (r)h^2)-\lambda.
\end{equation}
   As in the proof of Theorem \ref{t1},  let  $u\in C_0^\infty(\Sigma)$ arbitrary and let $v=u/h\in C_0^\infty(\Sigma)$. For $u=vh$, a similar computation yields
\begin{equation}\label{f2}
L_\phi[u] =v\cdot L_\phi[h]+h\Delta_\phi v+2\langle\nabla v,\nabla h\rangle.
\end{equation}
Analogously as in \eqref{l5}, using \eqref{f1} and \eqref{f2} we have
\begin{equation}\label{h4} 
\begin{split} 
Q_\phi[u]&=-\int_\Sigma \left( v^2  h \cdot L_\phi[h]+vh^2\Delta_\phi v+2v h  \langle\nabla v,\nabla h\rangle\right) \, d A_\phi\\
&=\int_\Sigma h^2\lvert \nabla v\rvert^2\, d A_\phi-\int_\Sigma   h v^2\cdot L_\phi [h] \, d A_\phi\\
&=\int_\Sigma h^2 \lvert \nabla v\rvert^2\, d A_\phi+4\int_\Sigma h^2v^2(\phi'(r)+\phi''(r) h^2) \, d A_\phi+\lambda\int_\Sigma hv^2\, dA_\phi\\
&=\int_\Sigma h^2\lvert \nabla v\rvert^2\, d A_\phi+4\int_\Sigma h^2v^2(\phi'(r)+\phi''(r) h^2) \, d A_\phi\geq 0,
\end{split}
\end{equation}
where we have used that $\lambda=0$ in the last identity.
\end{proof}

\begin{corollary} \label{ce2} Radial graph expanders   are strongly $\phi$-stable.  
\end{corollary}

As in Remark \ref{rem1}, monotonically increasing convex functions   $\phi$    satisfy  $\phi'(r)+\phi''(r)h^2\geq 0$ on $\Sigma$.

If $H_\phi=\lambda$ is a non-zero constant, the sign of the integral $\lambda\int_\Sigma h v^2\, dA_\phi$ in \eqref{l4} is nonnegative if  $\lambda$ and $h$ have the same sign. It follows from \eqref{h4} the following corollary.  

\begin{corollary}\label{c2}
Let $e^\phi$ be a radial density, with $\phi=\phi(r)$,  $r=\lvert x\rvert^2$. Suppose $\Sigma$ is a  $\phi$-stationary  radial graph  such that  $H_\phi$ and $h$ have the same sign.  If $\phi'(r)+ \phi''(r)h^2\geq 0$, then $\Sigma$ is strongly $\phi$-stable. 
\end{corollary}

\begin{remark} If the radial density is constant, $\phi$-stationary hypersurfaces are cmc hypersurfaces. Theorem \ref{t1} for cmc vertical graphs is well known \cite{fc}. However, as far as the author knows, Theorem \ref{t2} and Corollary \ref{c2}  have not been explicitly stated   in the literature of cmc hypersurfaces. The first results of existence of cmc radial graphs appeared in \cite{ra} and \cite[\S 23]{se}: see also \cite{cal,fu,lo0}.
\end{remark}

In the last part of this section we study densities depending only on a coordinate of $\r^{n+1}$. Let $\phi=\phi(t)$,   $t=x_{n+1}$.  Then 
\begin{equation}\label{s1}
\overline{\nabla}\phi=\phi'(t)\a_{n+1},\qquad \overline{\nabla}^2\phi(N,N)=\phi''(t)N_{n+1}^2.
\end{equation}
 If $\Sigma$ is $\phi$-stationary, then \eqref{n1} writes as
$$\Delta_\phi N_{n+1}+\lvert A\rvert^2 N_{n+1} =-\phi''(t)N_{n+1}\sum_{i}\langle e_i,\a_{n+1}\rangle^2=-\phi''(t)N_{n+1}(1-N_{n+1}^2).$$
Combining with the expression of the Bakry-\'Emery Ricci tensor in \eqref{s1}, we have
\begin{equation}\label{ff}
L_\phi[N_{n+1}]=-\phi'' (t)N_{n+1}.
\end{equation}

\begin{theorem} \label{t4}
Suppose $\phi=\phi(x_{n+1})$. If $\phi$ is convex, then $\phi$-stationary vertical  graphs are   strongly $\phi$-stable.  
\end{theorem}

\begin{proof} 
The proof is similar as in Theorem \ref{t1}. If    $u\in C_0^\infty(\Sigma)$ is arbitrary, let $v=u/N_{n+1}$. Using \eqref{ff} and the convexity of $\phi$,   we have
\begin{equation*}
\begin{split}
Q_\phi[u]&=  \int_\Sigma N_{n+1}^2\lvert \nabla v\rvert^2\, d A_\phi-\int_\Sigma  v^2N_{n+1}\cdot L_\phi[N_{n+1}] \, d A_\phi\\
&=\int_\Sigma N_{n+1}^2\lvert \nabla v\rvert^2\, d A_\phi+\int_\Sigma  \phi''(t) v^2N_{n+1}^2  \, d A_\phi\\
&\geq 0.
\end{split}
\end{equation*}
\end{proof}
 
 Two interesting densities under the conditions of Theorem \ref{t4} are the following. 
  
   \begin{corollary}  If $\alpha\leq 0$, $\alpha$-singular minimal vertical graphs are   strongly $\phi$-stable.  
\end{corollary}

  \begin{corollary} $\lambda$-translators  vertical graphs are   strongly $\phi$-stable.  
\end{corollary}

Notice that this result    extends the particular case of translating solitons proved in \cite{sh}.

 We now investigate the  version of Theorem \ref{t4} for radial graphs. From \eqref{s1}, we have $nH=\phi'(t)N_{n+1}+\lambda$, where $\lambda$ is a constant. Now \eqref{h3} is
 \begin{equation*} 
 \begin{split}
 \Delta_\phi h+\lvert A\rvert^2h&=-\phi'(t) N_{n+1}-\lambda-\phi''(t)N_{n+1}\sum_{ij}\langle e_i,\a_{n+1}\rangle\langle e_i,\a_j\rangle\langle x,\a_j\rangle\\
 &=-\phi' (t)N_{n+1}-\lambda-\phi''(t)N_{n+1}(x_{n+1}-hN_{n+1}).
 \end{split}
 \end{equation*}
 This identity together \eqref{s1} yields
 \begin{equation}\label{ff2}
 L_\phi[h]=-\phi'(t) N_{n+1}-\lambda-\phi''(t)N_{n+1}x_{n+1}.
 \end{equation}
Thanks to \eqref{ff2},  the   expression of $Q_\phi[u]$ in  \eqref{h4} yields
 \begin{equation}\label{ff3}
 \begin{split}
 Q_\phi[u]& = \int_\Sigma h^2 \lvert \nabla v\rvert^2\, d A_\phi-\int_\Sigma   h v^2\cdot L_\phi [h] \, d A_\phi\\
 &= \int_\Sigma h^2\lvert \nabla v\rvert^2\, d A_\phi+\lambda \int_\Sigma   h v^2  \, d A_\phi+\int_\Sigma   h v^2N_{n+1}(\phi'(t)+x_{n+1}\phi''(t))  \, d A_\phi.\\
 \end{split}
 \end{equation}
 
 Although the last integral is a bit cumbersome for a general density $\phi=\phi(x_{n+1})$, we have an interesting result for $\alpha$-singular minimal surfaces. 
 
  \begin{theorem}\label{t5}
  Independently of the sign of $\alpha$,     $\alpha$-singular minimal radial graphs such that  $H_\phi$ and $h$ have the same sign are strongly $\phi$-stable.  In particular, this includes the case of $\phi$-minimal radial graphs.  
\end{theorem}

\begin{proof}
Since $\phi(t)=\alpha\log(t)$, then $\phi'(t)+x_{n+1}\phi''(t)=0$, and the result is now a consequence of  \eqref{ff3}. 
\end{proof}
 
  Notice that for $\lambda$-translators, no analogous results are derived from the above computations. In such a case, $nH=N_{n+1}+\lambda$ and $\phi(t)=t$. Then \eqref{ff3} is  
\begin{equation*}
  Q_\phi[u]  =
  \int_\Sigma h^2\lvert\nabla v\rvert^2\, d A_\phi+\lambda \int_\Sigma   h v^2  \, d A_\phi+  \int_\Sigma   h v^2N_{n+1}  \, d A_\phi.
  \end{equation*}
 For example, if $\lambda=0$, a sufficient condition to ensure the non-negativity of $Q_\phi[u]$ is that $hN_{n+1}\geq 0$. Let us observe that a change of orientation on $\Sigma$ does not affect to the sign of $hN_{n+1}$.

%%%%%%%%%%%%%%%%%%%
\section{Minimizers of vertical graphs in manifolds with some densities}\label{sec3}
%%%%%%%%%%%%%%
 
In this section, and for special types of densities, we prove that vertical graphs  are minimizers in a suitable class of hypersurfaces.  Results of  weighted minimizers were obtained   in \cite{ca,h1} for $\phi$-minimal hypersurfaces and our results are inspired in them.  In these papers,   the authors used the language of calibrations based in the theory of minimal hypersurfaces \cite{mo1,mo2}. Since in the present paper, the manifold is $\r^{n+1}$ endowed with a  density, the arguments can be directly done without any reference to  calibrations.

Let $\Sigma$ be a compact  $\phi$-stationary graph on a domain $U$ of the hyperplane $\Pi$ of equation $x_{n+1}=0$ which we identify with $\r^n$. Consider the notation $x=(x_1,\ldots,x_{n+1})= (q,x_{n+1})$ for a generic point of $U\times\r$, where $q\in U$.  Suppose that $\Sigma$ is the graph of a function $f\in C^\infty(U)$, $f=f(x_1,...,x_n)$. Without loss of generality, we can assume that the Gauss map of $\Sigma$ is 
\begin{equation}\label{nnn}
N(q,f(q))= N(q)=\frac{1}{\sqrt{1+\lvert Df\rvert^2}}\left(-Df,1\right),
\end{equation}
where $Df\in\r^n$ is the Euclidean gradient of $f$. 

On the domain $U\times\r\subset\r^{n+1}$, define a vector field $X$ by translations of $N $ along the $x_{n+1}$-axis and multiplied  by  the density $e^\phi$. To be precise, let
$$X(x)=X(x_1,\ldots,x_{n+1})=e^{\phi(x)} N (q,f(q)).$$
Using \eqref{h}, we have 
\begin{equation}\label{d1}
\begin{split}
\mathrm{div}_{\r^{n+1}}(X)(x)&=-e^{\phi(x)}\mathrm{div}_{\r^n}\left(\frac{Df}{\sqrt{1+\lvert Df\rvert^2}}\right)(q)+e^{\phi(x)}\langle\overline{\nabla}\phi(x),N(q)\rangle\\
&=-nH(q,f(q)) e^{\phi(x)}+e^{\phi(x)}\langle\overline{\nabla}\phi(x),N(q)\rangle\\
&=e^{\phi(x)}\left(-\lambda- \langle\overline{\nabla}\phi(q,f(q)),N(q)\rangle+\langle\overline{\nabla}\phi(x),N(q)\rangle\right)\\
&=e^{\phi(x)}\left(  \langle\overline{\nabla}\phi(x)-\overline{\nabla}\phi(q,f(q)),N(q)\rangle -\lambda\right).
\end{split}
\end{equation}

Under some hypothesis on   $\phi$, we prove that $\Sigma$ minimizes the weighted area $A_\phi$ in a suitable class of hypersurfaces. This class of hypersurfaces     is  defined by
\begin{equation*}
\begin{split}
\mathcal{S}(\Sigma;U)=&\Big\{\widetilde{\Sigma}\hookrightarrow  \r^{n+1} \mbox{hypersurface}: \widetilde{\Sigma}\subset U\times\r,\ \partial\widetilde{\Sigma}=\partial\Sigma,\\
&\mbox{$\widetilde{\Sigma}$ and $\Sigma$ are in the same homology class,}  \\
& \mbox{$\widetilde{\Sigma}$ encloses the same weighted volume than $\Sigma$}\Big\}.
\end{split}
\end{equation*}
To simplify, we will say that $\Sigma$ is a weighted minimizer  in $\mathcal{S}(\Sigma;U)$.  
 The last condition about the weighted volume is natural because we are assuming that $\Sigma$ has constant weighted mean curvature. 

\begin{theorem}\label{estable1} Suppose that $e^\phi$ is a density depending only on $x_{n+1}$ and  that $\phi$ is a convex function. If  $\Sigma$ is a compact $\phi$-stationary vertical graph  on a domain $U$ of the   hyperplane $\Pi$ of equation $x_{n+1}=0$, then  $\Sigma$ is a weighted minimizer    in $\mathcal{S}(\Sigma;U)$.
\end{theorem}
\begin{proof}
Let   $\widetilde{\Sigma}$ be a   hypersurface in $\mathcal{S}(\Sigma;U)$ and we will prove that $A_\phi(\Sigma)\leq A_\phi(\widetilde{\Sigma})$. Since $\Sigma$ and $\widetilde{\Sigma}$ belong to the same homology class, both hypersurfaces determine an oriented $3$-chain $\Omega$ whose boundary is $\Sigma\cup\widetilde{\Sigma}$. Let $\widetilde{N}$ be the orientation on $\widetilde{\Sigma}$ compatible with the orientation of the $3$-chain $\Omega$.

Since the function $\phi$ depends only on the   $x_{n+1}$-coordinate, then $\overline{\nabla}\phi=\phi'(t)\a_{n+1}$. We can follow the computation of \eqref{d1} obtaining from \eqref{nnn}
$$\mathrm{div}_{\r^{n+1}}(X)(x)  =e^{\phi(x)}\left(\frac{\phi'(x_{n+1})-\phi'(f(q)) }{\sqrt{1+\lvert Df\rvert^2}}-\lambda \right).$$
The divergence theorem yields
\begin{equation}\label{e3}
\begin{split}
  \int_\Omega  \left( \frac{\phi'(x_{n+1})-\phi'(f(q)) }{\sqrt{1+\lvert Df\rvert^2}}-\lambda \right)e^\phi\, dV &= \int_\Sigma\langle X,N\rangle\, dA+\int_{\widetilde{\Sigma}}\langle X,\widetilde{N}\rangle\, d\widetilde{A}\\
&=\int_\Sigma e^\phi\, dA+\int_{\widetilde{\Sigma}} e^\phi\langle N,\widetilde{N}\rangle\, d\widetilde{A} \\
&\geq  \int_\Sigma e^\phi\, dA-\int_{\widetilde{\Sigma}}e^\phi\, d\widetilde{A}\\
&=A_\phi(\Sigma)-A_\phi(\widetilde{\Sigma})
\end{split}
\end{equation}
because $X=e^\phi N$ in $\Sigma$ and $\langle N,\widetilde{N}\rangle\leq 1$.  Since the enclosed weighted volume of $\Sigma$ and $\widetilde{\Sigma}$ is the same, we find $\int_\Omega e^\phi\, dV=\int_\Omega dV_\phi=0$. Thus
\begin{equation*}
A_\phi(\Sigma)-A_\phi(\widetilde{\Sigma})\leq  \int_\Omega \frac{\phi'(x_{n+1})-\phi'(f(q)) }{\sqrt{1+\lvert Df\rvert^2}}e^\phi\, dV.
  \end{equation*}
  The proof is achieved if we prove that $\phi'(x_{n+1})-\phi'(f(q))\leq 0$ in $\Omega$.    The volume element $dV$ in $\Omega$ is $dV=dx_1\wedge\ldots\wedge dx_{n+1}$ on each positively oriented component of the chain $\Omega$, that is, on each component for which $N$ points out of the component. On the other hand, $dV=- dx_1\wedge\ldots\wedge dx_{n+1}$ on each negatively oriented component. Since $\widetilde{\Sigma}$ is included in the cylinder $U\times\r$, we have $f(q)\geq x_{n+1}$ in the positively oriented components and $f(q)\leq x_{n+1}$ in the negatively oriented ones. Because $\phi$ is convex, we have that $\phi'$ is monotically increasing and we deduce that  $\phi'(x_{n+1})\leq \phi'(f(q))$ in the positively oriented components and $\phi'(x_{n+1})\geq \phi'(f(q))$ in the negatively oriented ones. This concludes  the proof.
\end{proof}

Examples of densities under the hypothesis of Theorem \ref{estable1} are $e^\phi=x_{n+1}^\alpha$ when $\alpha\leq 0$. In such a case, 
$\phi''(t)=-\alpha/t^2\geq 0$. See \cite{lo}.

\begin{corollary} If $\alpha\leq 0$,  then any compact $\alpha$-singular minimal vertical graph   on a domain $U$ of the   hyperplane $\Pi$ of equation $x_{n+1}=0$ is weighted minimizer  in $\mathcal{S}(\Sigma;U)$.  
\end{corollary}
When $\alpha=0$, the statement is the known result of the theory of constant mean curvature hypersurfaces. 

Another examples of densities correspond with $\lambda$-translators because $\phi(t)=t$ and $\phi''=0$. Then we have the following result proved in \cite{h1}.

\begin{corollary} Compact $\lambda$-translators which are graphs on a domain $U$ of the   hyperplane $\Pi$ of equation $x_{n+1}=0$ are weighted minimizer  in $\mathcal{S}(\Sigma;U)$.  
\end{corollary}

Notice that for radial densities, identity \eqref{d1} is difficult to manage. However, in the case of expanders, it follows the next result.

\begin{theorem}\label{estable2} Compact   graph expanders   on a domain $U$ of the   hyperplane $\Pi$ of equation $x_{n+1}=0$ are weighted minimizers in $\mathcal{S}(\Sigma;U)$.
\end{theorem}

\begin{proof} For expanders, $\phi(\lvert x \rvert^2)=|x|^2/4$ and $\overline{\nabla}\phi=x/2$. Hence 
$$\langle \overline{\nabla}\phi(x)-\overline{\nabla}\phi(q,f(q)),N(q)\rangle=\frac12\langle (x_{n+1}-f(q))\a_{n+1},N(q)\rangle= \frac{x_{n+1}-f(q)}{2\sqrt{1+|Df|^2}}.$$
From \eqref{d1}, we have 
$$\mathrm{div}_{\r^{n+1}}(X)(x) =e^{\phi(x)}\left( \frac{x_{n+1}-f(q) }{2\sqrt{1+\lvert Df\rvert^2}}-\lambda \right).$$
We now follow the same steps   of   the proof of Theorem \ref{estable1}. The divergence theorem gives
\begin{equation*}
A_\phi(\Sigma)-A_\phi(\widetilde{\Sigma})\leq  \int_\Omega \frac{x_{n+1}-f(q) }{2\sqrt{1+\lvert Df\rvert^2}}e^\phi\, dV.
  \end{equation*}
  Finally, we also have $x_{n+1}-f(q)\leq 0$ in $\Omega$. This proves $A_\phi(\Sigma)- A_\phi(\widetilde{\Sigma})\leq 0$, obtaining the result.
\end{proof}

Following with the same ideas, we revisit a result proved in \cite{ca,h1} for $\phi$-minimal hypersurfaces.

\begin{theorem}\label{t-estable3}  Suppose that  the density $e^\phi$ depends only on the variables $(x_1,\ldots,x_{n})$.  If $\Sigma$ is a compact $\phi$-stationary graph  on a domain $U$ of the hyperplane $\Pi$ of equation $x_{n+1}=0$, then $\Sigma$ is a weighted minimizer in $\mathcal{S}(\Sigma;U)$
\end{theorem} 

\begin{proof} Now $\overline{\nabla}\phi(x)-\overline{\nabla}\phi(q,f(q))=0$. Identity \eqref{d1} reduces into
\begin{equation}
\mathrm{div}_{\r^{n+1}}(X)(x) =-\lambda e^{\phi(x)},
\end{equation}
and the result is now proved. \end{proof}

\section*{Acknowledgements}    
The author  has been partially supported by MINECO/MICINN/FEDER grant no. PID2023-150727NB-I00,  and by the ``Mar\'{\i}a de Maeztu'' Excellence Unit IMAG, reference CEX2020-001105- M, funded by MCINN/AEI/10.13039/ 501100011033/ CEX2020-001105-M.

%%%%%%%%%%%%%%%%%%
%\def\refname{References}
\bibliographystyle{amsplain}

\end{document}